\documentclass{amsart}
\usepackage{amsxtra}
\usepackage{mathrsfs}

\usepackage{amssymb,amsmath,amsthm,color}

\usepackage[bookmarksopen=true,bookmarksnumbered=true, bookmarksopenlevel=2, colorlinks=true,linkcolor=mblue,
citecolor=mblue,urlcolor=mblue, pdfstartview=FitH]{hyperref}

\definecolor{mblue}{rgb}{0,0,.8}

\newtheorem{theorem}{Theorem}

\newtheorem{remark}{Remark}

\newtheorem{prop}{Proposition}
\newtheorem{lemma}{Lemma}

\newcommand{\ZZ}{\ensuremath{\mathbb{Z}}}
\newcommand{\QQ}{\ensuremath{\mathbb{Q}}}
\newcommand{\QQbar}{\overline{\QQ}}

\newcommand{\PP}{\ensuremath{\mathbb{P}}}

\newcommand{\Gal}{\operatorname{Gal}}

\DeclareMathOperator{\Ind}{Ind}

\DeclareMathOperator{\Res}{Res}

\title{Quadratic twists of rigid Calabi-Yau threefolds over $\QQ$}

\author{Fernando Q.\ Gouv\^ea, Ian Kiming, Noriko Yui}

\address{F.\ Q.\ Gouv\^ea: Department of Mathematics, Colby College,
Waterville, ME 04901, USA.  E-mail: fqgouvea@colby.edu}

\address{I.\ Kiming: Department of Mathematics, University of
Copenhagen, Universitets\-parken 5, DK 2100 $\emptyset$, Copenhagen, Denmark.
E-mail: kiming@math.ku.dk}

\address{N.\ Yui: Department of Mathematics and Statistics,
Queen's University, Kingston, Ontario Canada K7L 3N6. Email:
yui@mast.queensu.ca}

\thanks{This work was partially supported by the Natural Sciences and
Engineering Research Council of Canada (NSERC), and by The Danish Council
for Independent Research}


\begin{document}

\begin{abstract} We consider rigid Calabi--Yau threefolds defined over
  $\QQ$ and the question of whether they admit quadratic twists. We give a
  precise geometric definition of the notion of a quadratic twists in this
  setting. 
  
  Every rigid Calabi--Yau threefold over $\QQ$ is modular so there is
  attached to it a certain newform of weight $4$ on some $\Gamma_0(N)$. We
  show that quadratic twisting of a threefold corresponds to twisting the
  attached newform by quadratic characters and illustrate with a number of
  obvious and not so obvious examples.
  
  The question is motivated by the deeper question of which newforms of
  weight $4$ on some $\Gamma_0(N)$ and integral Fourier coefficients arise
  from rigid Calabi--Yau threefolds defined over $\QQ$ (a geometric
  realization problem).
\end{abstract}

\maketitle

\section{Introduction}\label{intro}

Suppose $X$ is a rigid Calabi--Yau threefold defined over $\QQ$. As
Gouv\^ea and Yui observe in \cite{GY} (see also \cite{Di2010},
\cite{DiMa}), it follows from work of Khare and Winterberger that $X$
is modular: The $L$-series of $X$ coincides with the $L$-series of a
certain newform $f$ of weight $4$ on some $\Gamma_0(N)$. Alternatively,
there is a newform $f$ with integer coefficients such that, for any prime
$\ell$, the $\ell$-adic representation of $G_{\QQ} = \Gal(\QQbar /\QQ)$ on
$H^3(\bar{X},\QQ_{\ell})$ is isomorphic to the $\ell$-adic representation
of $G_{\QQ}$ attached to $f$.

Very little seems to be known about the form $f$. Notably, the
relation between the conductor $N$ and the geometry of $X$ still seems
to be poorly understood. (See the discussions and conjectures of
section 6.4 of \cite{Meyer}, as well as the paper \cite{Di}.)

Another unresolved and probably very hard question is the following. The
form $f$ above obviously has integral Fourier coefficients. Can one
conversely characterize the newforms of weight $4$ on some $\Gamma_0(N)$
with integral coefficients that arise from rigid Calabi--Yau threefolds
over $\QQ$? Do all such forms arise from Calabi-Yau threefolds? (This is a
kind of the geometric realization problem. See \cite{elkieschuett} for
the case of ``singular'' K3 surfaces and forms of weight $3$.)

A very weak version of this question is the topic of this paper: Given
a rigid Calabi--Yau threefold $X$ with form $f$ as above, for any non-square
rational number $d$ there is a twist $f_d$ of $f$ by the quadratic character
corresponding to the quadratic extension $K=\QQ(\sqrt{d})$ over
$\QQ$. This $f_d$ is again of the above form and so we can ask whether
$f_d$ arises from a rigid Calabi--Yau threefold $X_d$ over $\QQ$. This
will be the case whenever $X$ admits a quadratic twist by $d$ in the
sense we discuss next.

\section{Quadratic twists of rigid Calabi--Yau threefolds}\label{twists}

Let $X$ be a rigid Calabi--Yau threefold defined over $\QQ$. Suppose
that $d\in\QQ^{\times}$ is a squarefree integer, let $K :=
\QQ(\sqrt{d})$, and let $\sigma$ be the non-trivial automorphism of
$K/\QQ$. We say that a rigid Calabi--Yau threefold $X_d$ defined over
$\QQ$ is a {\it twist of $X$ by $d$} if there exist:

\begin{itemize}
\item an involution $\iota$ of $X_{\QQ}$ that acts as $-1$ on
  $H^3(\bar{X},\QQ_{\ell})$ for some prime $\ell$, and
\item an isomorphism $\theta \colon ~(X_d)_{K}
  \stackrel{\cong}{\rightarrow} X_{K}$ defined over $K$
\end{itemize}
such that:
$$
\theta^{\sigma} \circ \theta^{-1} = \iota.
$$

Notice that the condition $\iota = \theta^{\sigma} \circ \theta^{-1}$
necessarily implies that the involution $\iota$ satisfies
$\iota^{\sigma} = \iota$, i.e., that $\iota$ is defined over $\QQ$.
Conversely, given an involution $\iota$ on $X_\QQ$, one can always
find the isomorphism $\theta$. One takes the quotient of $X_K=X_\QQ
\otimes K$ by $\iota\otimes\sigma$, checks that it is defined over
$\QQ$ and that it is the twist $X_d$ as above. This will become clear
in the examples below: whenever we can find the appropriate involution
we can also construct a twist.
\smallskip

Since $X$ is a rigid Calabi--Yau threefold, $H^{3,0}(X)$ is
one-dimensional so there is a unique (up to scalar) holomorphic
$3$-form $\Omega$ on $X$. The involution $\iota$ should act on
$\Omega$ non-symplectically, sending it to $-\Omega$. Conversely,
since $X$ is rigid we have $h^{2,1}(X)=0$, so if $\iota$ sends
$\Omega$ to $-\Omega$ we see that $\iota$ acts as $-1$ on all of
$H^3(\bar{X},\QQ_{\ell})$. (Here the rigidity of $X$ is used in an
essential way.)

This is the method that we will primarily employ in the examples below
to ensure this part of the condition on the involution $\iota$.

\smallskip

One could envision relaxing the above definition in the direction of
just requiring the existence of an algebraic correspondence between
$(X_d)_{K}$ and $X_{K}$ and still retain (a somewhat stronger version
of) the theorem below. However, in the examples that we will give, we
actually find isomorphisms in all cases and have hence chosen to work
with the above definition.

The principles of proof of the following theorem should be well-known,
but we provide the details because of lack of a precise reference.

Recall that, given a newform $f$ of some weight and a non-square
$d\in\QQ$ there is a twist $f_d$ of $f$ by $d$ which is again a
newform of the same weight as $f$ (but potentially at another level)
and whose attached $\ell$-adic Galois representation (for some prime
$\ell$ and hence for all primes $\ell$) is isomorphic to the
$\ell$-adic representation attached to $f$ twisted by the quadratic
character $\chi$ corresponding to $K/\QQ$. If the Fourier coefficients
of $f$ and $f_d$ are $a_n$ and $b_n$, respectively, we have the
relation $b_p = \chi(p) a_p$ for almost all primes $p$. In particular,
since $\chi$ is quadratic, if $f$ has coefficients in $\ZZ$ then so
does $f_d$.

\begin{theorem}\label{thm_twist}
In the above setting, suppose that the newform (of weight $4$)
attached to $X$ is $f$. Then, if $X_d$ is a twist by $d$ of $X$ the
newform attached to $X_d$ is $f_d$, the twist of $f$ by the Dirichlet
character $\chi$ corresponding to the quadratic extension $K =
\QQ(\sqrt{d})$ of $\QQ$.

If we keep all hypotheses above except possibly that $\iota$ acts as
$-1$ on $H^3(\bar{X},\QQ_{\ell})$, we can still deduce that the
newform attached to $X_d$ is either $f$ or $f_d$.
\end{theorem}

\begin{proof}
Fix a prime number $\ell$, and consider the $\ell$-adic Galois
representations $\rho$ and $\rho_d$ attached to $X$ and $X_d$,
respectively: these are given by the action of $G_{\QQ} = \Gal(\QQbar
/\QQ)$ on the $2$-dimensional $\QQ_{\ell}$-vector spaces $V :=
H^3(\bar{X},\QQ_{\ell})$ and $V_d := H^3(\bar{X_d},\QQ_{\ell})$,
respectively.

Now, the newform $f$ attached to $X$ is determined uniquely by the
requirement that its attached $\ell$-adic representation be isomorphic
to $\rho$. Similarly, the newform attached to $X_d$ is determined by
the requirement that its attached $\ell$-adic representation be
isomorphic to $\rho_d$.

Since the $\ell$-adic representation attached to $f_d$ is isomorphic
to the twist by $\chi$ of the one attached to $f$, we see that what we
have to prove boils down to:
$$
\rho_d \cong \rho \otimes \chi .
$$

Put $N := G_K = \Gal(\QQbar /K)$ so that $N$ is a normal subgroup of
$G_{\QQ}$. The existence of the isomorphism $\theta \colon ~(X_d)_{/K}
\cong X_{/K}$ defined over $K$ translates into the existence of a
$\QQ_{\ell}$-linear isomorphism $V_d \rightarrow V$ commuting with the
action of $G_K$. I.e., in matrix terms we have an invertible matrix
$A$ with:
$$
\rho(n) A = A \rho_d(n) \quad \mbox{for all $n\in N$}.
$$

In matrix terms the conjugate isomorphism $\theta^{\sigma}$ of $V_d$
onto $V$ is then given by the matrix
$$
\rho(\sigma) A \rho_d(\sigma)^{-1} ;
$$
notice that we have here viewed $\sigma$ as an element of $G_{\QQ}$
via choice of a representative; the expression $\rho(\sigma) A
\rho_d(\sigma)^{-1}$ does not depend on this choice.

If now $\iota$ acts as $-1$ on $H^3(\bar{X},\QQ_{\ell})$ we can deduce
that the matrix
$$
\rho(\sigma) A \rho_d(\sigma)^{-1} A^{-1}
$$
is a non-trivial involution.
\smallskip

Define the representation $\rho'$ of $G_{\QQ}$ by $\rho' := A^{-1}
\rho A$ so that $\rho'(n) = \rho_d(n)$ for $n\in N$. Then, for
arbitrary $g\in G_{\QQ}$ and $n\in N$ we have
\begin{eqnarray*}
\rho_d(g) \rho'(n) \rho_d(g)^{-1} &=& \rho_d(g) \rho_d(n) \rho_d(g)^{-1} \\
&=& \rho_d(gng^{-1}) = \rho'(gng^{-1}) = \rho'(g) \rho'(n) \rho'(g)^{-1}
\end{eqnarray*}
so that
$$
\rho'(g)^{-1} \rho_d(g) \rho'(n) = \rho'(n) \rho'(g)^{-1} \rho_d(g)
$$
i.e., for any $g\in G_{\QQ}$ the matrix $\rho'(g)^{-1} \rho_d(g)$
commutes with all matrices $\rho'(n)$, $n\in N$.

Now, suppose first that $\rho$ (and hence $\rho'$) is absolutely
irreducible when restricted to $N$. In that case we deduce that
$\rho'(g)^{-1} \rho_d(g)$ is a scalar matrix, say with diagonal entry
$\mu(g)$. We have $\mu(n) = 1$ for $n\in N$ and see that $g\mapsto
\mu(g)$ is in fact a character of $G_{\QQ}$ factoring through $N =
G_K$. So, either $\mu = 1$ or $\mu = \chi$.

If we had $\mu = 1$ we would have $A^{-1} \rho(g) A = \rho'(g) =
\rho_d(g)$ for all $g\in G_{\QQ}$ and so in particular the matrix
$$
\rho(\sigma) A \rho_d(\sigma)^{-1} A^{-1}
$$
would be trivial. As we noted above, this can not happen if $\iota$
acts as $-1$ on $H^3(\bar{X},\QQ_{\ell})$. Hence, in that case we must
have $\mu = \chi$ and so $\rho_d = \rho' \otimes \chi \cong \rho
\otimes \chi$, as desired.
\smallskip

Suppose now that $\rho$ is not absolutely irreducible when restricted
to $G_K$. The same is then true of $\rho'$ and $\rho_d$. In this case
it is known, cf.\ (4.4), (4.5) of \cite{ribet}, that $\rho'$ is
induced from the $\ell$-adic representation $\psi$ attached to a
Gr\"ossencharacter over $K$: $\rho' = \Ind_{K/\QQ}(\psi)$, and
$\rho'_{\mid G_K}$ splits up as the sum of the two characters $\psi$
and $\psi^{\sigma}$. Notice that $\Ind_{K/\QQ}(\psi) =
\Ind_{K/\QQ}(\psi^{\sigma})$. Since $\rho'$ and $\rho_d$ have the same
restriction to $G_K$ we may then conclude that in fact $\rho' =
\rho_d$ as representations of $G_{\QQ}$, and hence that the newform
attached to $X_d$ is $f$. Furthermore, the matrix $\rho(\sigma) A
\rho_d(\sigma)^{-1} A^{-1}$ must then be trivial, and so we see that
this case in fact does not materialize if $\iota$ acts as $-1$ on
$H^3(\bar{X},\QQ_{\ell})$.
\end{proof}

\begin{remark}
What we have proved, in fact, is that if $\iota$ is nontrivial the
Galois representation on the middle cohomology of $X_d$ is isomorphic
to the tensor product of the representation on the middle cohomology
of $X$ and the one-dimensional Galois representation corresponding to
$K=\QQ(\sqrt d)$. For rigid Calabi--Yau manifolds, we know these
representations correspond to modular forms, but the question can, of
course, be asked without knowing anything about modularity.  We are
grateful to the referee for pointing this out to us.

\end{remark}

\subsection{Easy examples of twists}

The standard, simple example of twisting is of course for an elliptic
curve $E$ over $\QQ$, say given by a Weierstrass equation $y^2 = x^3 +
ax^2 + bx + c$. The twisted curve $E_d$ is then given by the equation
$dy^2 = x^3 + ax^2 + bx + c$ with the isomorphism $\theta : ~ E_d
\rightarrow E$ defined over $K = \QQ(\sqrt{d})$ by $\theta (x,y) =
(x,\sqrt{d} y)$. The corresponding involution $\iota$ is
$(x,y) \mapsto (x,-y)$. It is clear that $\iota$ sends the holomorphic
$1$-form $\Omega =\displaystyle\frac{dx}{y}$ to $-\Omega$.
\smallskip

For a number of rigid Calabi--Yau threefolds over $\QQ$ we can display
twists by essentially the same method: Consider for examples the
various cases of double octic Calabi--Yau threefolds over $\QQ$ (see
\cite{Meyer} for instance for a good overview). They are defined as
hypersurfaces of the form
$$
y^2=f_8(x_1,x_2,x_3,x_4)
$$
where $f_8$ is a degree $8$ homogeneous polynomial. As in the case of
elliptic curves, we have an obvious twist given by
$$
dy^2=f_8(x_1,x_2,x_3,x_4) .
$$

The corresponding involution is of course again given by $y\mapsto
-y$. Again it is clear that $\iota$ sends the holomorphic $3$-form
$$
\Omega=\displaystyle\frac{\sum_{i=1}^4 (-1)^i x_i\, dx_1 \wedge\cdots
  \widehat{\wedge{dx_i}}\wedge\cdots\wedge dx_4}{y}
$$
to $-\Omega$.

This is also completely analogous to the case certain modular double
sextic $K3$ surfaces, see \cite{P}. These have form
$$
w^2=f_6(x,y,z)
$$
where $f_6$ is a projective smooth curve of degree $6$. As above, we
get a twist of this surface (in the sense analogous to Theorem
\ref{thm_twist}) via the twisted equation
$$
dw^2=f_6(x,y,z)
$$
for a non-square rational number $d$. The involution is again given by
$w\mapsto -w$. The holomorphic $2$-form
$$
\Omega=\displaystyle\frac{z\,dx\wedge dy-x\,dy\wedge dz+y\,dx\wedge dz}{w}
$$
is sent by $\iota$ to $-\Omega$.

\subsection{Self-fiber products of rational elliptic surfaces with
  section and their twists}\label{verrill}

Slightly more complicated examples arise in connection with the rigid
Calabi--Yau threefolds studied by H.\ Verrill in the appendix to
\cite{Y03}: She determined the $L$-series via the point counting
method for the six isomorphism classes of rigid Calabi--Yau threefolds
constructed as self-fiber products of rational elliptic surfaces with
section by Schoen \cite{S}. Along the way, she discussed twists by
quadratic characters.

These six rigid Calabi--Yau threefolds over $\QQ$ are defined as
follows: Start with semi-stable families of elliptic curves $\pi:
\mathcal{Y}\to\mathbb{P}^1$, i.e., $\mathcal{Y}$ is a smooth surface
and the singular fibers have type $I_m$. Beauville \cite{B} gave a
complete list of these families. These are realized as the resolutions
of singular surfaces
$\bar{\mathcal{Y}}\subset\mathbb{P}^2\times\mathbb{P}^1$ given by the
following equations:
$$
\begin{array}{ccl}
\# & \mbox{Equation for $\bar{\mathcal{Y}}$} & \\ \hline
I & (x^3+y^3+z^3)\mu & =\lambda xyz \\
II & x(x^2+z^2+2zy)\mu & =\lambda (x^2-y^2)z \\
III & x(x-z)(y-z)\mu & =\lambda (x-y)yz \\
IV & (x+y+z)(xy+yz+zx)\mu & =\lambda xyz \\
V & (x+y)(xy-z^2)\mu & =\lambda xyz \\
VI & (x^2y+y^2z+z^2x)\mu & =\lambda xyz \\
\end{array}
$$

The fibration $\bar\pi: \bar{\mathcal{Y}}\to\mathbb{P}^1$ is given by
projecting to $\mathbb{P}^1$, and $\mathcal{Y}$ is obtained by
resolving $\bar{\mathcal{Y}}$. Now take the self-fiber product
$\mathcal{Y}\times_{\mathbb{P}^1}\mathcal{Y}$. Schoen \cite{S} shows
that a small resolution exists and that the resulting smooth variety
$X$ is a rigid Calabi--Yau threefold defined over $\QQ$. Thus, in each
case there is a newform of weight $4$ attached to $f$. In each case,
the form was identified by Verrill via determination of the $L$-series
of $X$ (point counting.) Here is the table of newforms from Verrill.

$$\begin{array}{cccc}
\# & Newform & \mbox{modular group} & \mbox{level} \\ \hline
I & \eta(q^3)^8 & \Gamma(3) & 9 \\
II & \eta(q^2)^4\eta(q^4)^4 & \Gamma_1(4)\cap \Gamma(2) & 8 \\
III & \eta(q)^4\eta(q^5)^4 & \Gamma_1(5) & 5 \\
IV & \eta(q)^2\eta(q^2)^2\eta(q^3)^2\eta(q^6)^2 & \Gamma_1(6) & 6 \\
V & \eta(q^4)^{16}\eta(q^8)^{-4}\eta(q^2)^{-4} & \Gamma_0(8)\cap \Gamma_1(4)
& 16 \\
VI & \eta(q^3)^8 & \Gamma_0(9)\cap \Gamma_1(3) & 9 \\
\end{array}
$$

In each case, one can display a twist $X_d$ of $X$ so that $X_d$
corresponds to twisting the newform by the quadratic character
belonging to $\QQ(\sqrt{d})/\QQ$. Consider for instance type $V$
above. Given a non-square $d\in \QQ^{\times}$ let $X_d$ be the variety
arising from the equation
$$
(x+y)(xy-dz^2)\mu  =\lambda xyz
$$
by a process analogous to the one leading to $X$ above.

Then we have an isomorphism $\theta \colon X_d \rightarrow X$ defined
over $\QQ(\sqrt{d})$ and given by
$$
((x:y:z),(\mu:\lambda)) \mapsto ((\sqrt{d} x: \sqrt{d}
y:z),(\mu:\sqrt{d} \lambda)).
$$

In the setup of Theorem \ref{thm_twist}, the involution $\iota$ is
given by
$$
\iota((x:y:z),(\mu:\lambda)) = ((-x:-y:z),(\mu:-\lambda)) .
$$

That $X_d$ is a genuine twist of $X$, i.e., that the attached newform
is $f_d$ rather than $f$ can be ascertained via point counting,
cf.\ appendix in \cite{Y03}.

The other examples can be dealt with in similar fashions.

%
%
%

\subsection{The Schoen quintic and its quadratic twists}\label{schoen}

As a more interesting test case, we consider the Schoen quintic
$$
x_0^5 + x_1^5 + x_2^5 + x_3^5 + x_4^5 = 5x_0x_1x_2x_3x_4.
$$

We write
$$
f=f(x_0,x_1,x_2,x_3,x_4)=x_0^5+x_1^5+x_2^5+x_3^5+x_4^5-5x_0x_1x_2x_3x_4 = 0.
$$

This is a singular threefold with $125$ nodes (ordinary double points) as
only singularities, and a small resolution of singularities produces a
rigid Calabi-Yau threefold $X$ that is known, cf.\ \cite{S_quintic}, to be
associated to a newform of weight $4$ and level $25$ (the modular form
$25k4A1$); see also \cite{Meyer}, section 3.1.

We seek an involution $\iota$ of $X$ that acts on $H^{3,0}(X)$ as
multiplication by $-1$. Since $H^{3,0}(X)$ is generated by a unique
holomorphic $3$-form $\Omega$ (up to scalar), $\iota$ should send $\Omega$
to $-\Omega$.

To determine the action of $\iota$ on $\Omega$ we can use either of the
following two arguments:

According to Cox and Katz \cite{CK}, especially section 2.3 and the formula
(2.7) therein, $\Omega$ can be computed on the smooth part as
$$
\Omega=\mbox{Res}(\frac{\omega}{f})
$$
where
$$
\omega=\sum_{i=0}^4 (-1)^ix_i\, dx_0\cdots \wedge
\widehat{dx_i}\wedge\cdots dx_4
$$
and where `Res' denotes Poincar\'{e} residue.

Alternatively, it follows from Lemma \ref{diff_form} below that we have a
holomorphic $3$-form
$$
\frac{dx_0 \wedge dx_1 \wedge dx_2}{\partial f /\partial x_3}.
$$
on the Zariski open set where $x_4$ and $\partial f /\partial x_3$ are
both non-vanishing, and that this extends to the Calabi-Yau threefold $X$.
\bigskip

Can one construct the requisite quadratic twists of the Schoen quintic? In
Gouv\^ea and Yui \cite {GY} it was briefly asserted that quadratic twist
indeed exists for the Schoen quintic. We now discuss details of this claim.

\begin{prop}
For any non-square $d\in\QQ^{\times}$ the Schoen quintic has a twist
by $d$. The corresponding involution $\iota$ defined over $\QQ$ is
given explicitly on the coordinates by
$$
\iota : (x_0,x_1,x_2,x_3,x_4)\mapsto (x_1,x_0,x_2,x_3,x_4),
$$
and sends $\Omega\in H^{3,0}(X)$ to $-\Omega$.
\end{prop}

\begin{proof} First, it is plain that $\iota$ sends the above $\Omega$ to
$-\Omega$ (using any of the two descriptions of $\Omega$.) 
  
Put $u=x_0+x_1$ and $v=x_0-x_1$. Then the equation for the quintic
equation can be written as a polynomial in $u$ and $v^2$ as follows:
$$
u^5+10u^3v^2+5uv^4+16(x_2^5+x_3^5+x_4^5)-20(u^2-v^2)x_2x_3x_4=0.
$$

Replacing $v$ by $\sqrt{d} v$, we obtain the following quintic equation:
$$
u^5+10du^3v^2+ 5d^2 uv^4 + 16(x_2^5+x_3^5+x_4^5)-20(u^2-dv^2)x_2x_3x_4=0 ,
\leqno{(\ast)}
$$
and we see how to apply Theorem \ref{thm_twist}: The equation
$(\ast)$ gives rise to a rigid Calabi--Yau threefold $X_d$ defined
over $\QQ$. Then we have an isomorphism $\theta \colon ~ X_d
\rightarrow X$ defined over $\QQ(\sqrt{d})$ and given by
$$
(u,v,x_2,x_3,x_4) \mapsto (u,\sqrt{d} v,x_2,x_3,x_4)
$$
so that $\theta^{\sigma} \circ \theta^{-1}$ is the involution given by
$(u,v,x_2,x_3,x_4) = (u,-v,x_2,x_3,x_4)$. This is precisely the
involution $\iota$ so the existence of the twist follows from Theorem
\ref{thm_twist}.
\end{proof}

\subsection{Explicit description for a holomorphic $3$-form for a complete
  intersection Calabi--Yau threefold} Before we go into further examples,
we give an explicit description of a holomorphic $3$-form for a complete
intersection Calabi--Yau threefold, by the Griffiths residue theorem or its
generalized version. We are grateful to Bert van Geemen for communicating
to us the following lemma as well as its proof. 

\begin{lemma}\label{diff_form} Let $Y=V(f_1,\cdots,f_k)$ be a complete
  intersection in $\PP^n$ of dimension $d=:n-k$ where $f_1,\ldots, f_k$ are
  homogeneous equations in the homogeneous variables $x_0,\ldots,
  x_n$. Assume that $Y$ is a normal crossings divisor. 

Let $i_0\in \{0,\ldots,n\}$, let $I\subseteq \{0,\ldots,n\} \backslash \{
i_0\}$ have cardinality $k$, and consider 
$$
D_I := \det \left(\frac{\partial f_i}{\partial x_j}\right)_{\substack{1\leq
    i\leq k \\ j\in I}} 
$$

Then, on the Zariski open set where $x_{i_0}$ and $D_I$ are both
non-vanishing, a holomorphic $d$-form is given by 
$$
\Omega= \frac{\bigwedge_{j\in \{0,\ldots,n\} \backslash (\{ i_0\} \cup I)}
  dx_j}{D_I} 
$$

If additionally $Y$ has a crepant resolution $X$ that is Calabi-Yau variety
of dimension $\dim X\leq 3$, then $\Omega$ extends to all of $X$. 
\end{lemma}

Let us remind that a `crepant resolution' is one that does not change the
canonical class, cf.\ \cite{reid}, \S $2$. 

The Lemma applies to the Schoen quintic as well as the threefolds that we
shall consider below because the singularities involved are ordinary double
point in all cases. 

The proof is given below in the appendix (section \ref{appendix}.)

\subsection{Two rigid Calabi--Yau threefolds of Werner and van
  Geemen} \label{wvg} Werner and van Geemen \cite{vG} constructed a
number of examples of rigid Calabi--Yau threefolds over $\QQ$. They
are complete intersection Calabi--Yau threefolds.

We consider two of them. First, the rigid Calabi--Yau threefold
denoted by $\tilde{V}_{33}$: Let $V_{33}\subset\mathbb{P}^5$ be the
threefold defined by the system of equations
$$
\begin{array}{ccccccccccccl}
x_0^3&+&x_1^3&+&x_2^3&+&x_3^3&&&&&=& 0 \\
&&&&            x_2^3&+&x_3^3&+&x_4^3&+&x_5^3&=&0 .\\
\end{array}
$$

$V_{33}$ has $9$ singularities, and let $\tilde{V}_{33}$ be the blow up of
$V_{33}$ along its singular locus (big resolution). Then $\tilde{V}_{33}$
is a rigid Calabi--Yau threefold over $\QQ$, and it is modular with a
corresponding newform $f$ of weight $4$ on $\Gamma_0(9)$: 
$$
f(q)=\eta(q^3)^8.
$$

It is shown by Kimura \cite{K} that if $E\subset\PP^2$ is the curve
defined by $x_0^3+x_1^3+x_2^3=0$, then there is a dominant rational
map from $E^3$ to $V_{33}$ of degree $3$. Consequently, the $L$-series
coincide. By Lemma \ref{diff_form}, a holomorphic $3$-form of $V_{33}$ is
given in affine coordinates by 
$$
\Omega=\frac{d x_2\wedge dx_3\wedge dx_4}{x_1^2x_5^2}.
$$

\begin{prop} For any non-square $d\in\QQ^{\times}$ the rigid Calabi--Yau
  threefold $\tilde{V}_{33}$ has a twist $\tilde{V}_{33,d}$ by $d$. The
  corresponding involution $\iota$ is defined by permuting $x_2$ and
  $x_3$. 
\end{prop}
\begin{proof} Put $u=x_2+x_3,\, v=x_2-x_3$. Then the equation for $V_{33}$
  can be expressed in terms of $x_0,x_1,x_4,x_5$ and $u$ and $v^2$: 
  $$
\begin{array}{ccccccccccccl}
4x_0^3&+&4x_1^3&+&u^3&+&3uv^2&&&&&=& 0 \\
&&&&             u^3&+&3uv^2&+&4x_4^3&+&4x_5^3&=&0 .\\
\end{array}
$$

Replacing $v$ by $\sqrt{d} v$ in this system we obtain a system of
equations that gives rise to $\tilde{V}_{33,d}$. Applying Theorem
\ref{thm_twist} shows that $\tilde{V}_{33,d}$ is twist by $d$ of
$\tilde{V}_{33}$ with the corresponding involution given by $v\mapsto
-v$, i.e., by $(x_2,x_3) \mapsto (x_3,x_2)$.

The holomorphic $3$-form $\Omega$ above clearly changes sign when
$x_2$ and $x_3$ are interchanged.
\end{proof}

Secondly, we can consider the rigid Calabi--Yau threefold denoted by
$\tilde{V}_{24}$: Let $V_{24}\subset\mathbb{P}^5$ be the threefold
defined by the equations:
$$
\begin{array}{ccccccccccccc}
x_0^2&+&x_1^2&+&x_2^2&-&x_3^2&-&x_4^2&-&x_5^2&=& 0 \\
x_0^4&+&x_1^4&+&x_2^4&-&x_3^4&-&x_4^4&-&x_5^4&=&0 .\\
\end{array}
$$
Then $V_{24}$ has $122$ nodes (ordinary double points) as
only singularities.  Let $\tilde{V}_{24}$ be the blow up of $V_{24}$
along its singular locus (small resolution).
Then $\tilde{V}_{24}$ is a rigid Calabi--Yau threefold over
$\QQ$, and it is modular with a corresponding newform $g$ which is the
newform of weight $4$ on $\Gamma_0(12)$.

\begin{prop} For any non-square $d\in\QQ^{\times}$ the rigid
  Calabi--Yau threefold $\tilde{V}_{24}$ has a twist
  $\tilde{V}_{24,d}$ by $d$. The corresponding involution $\iota$ is
  given by:
$$
\iota: x_1\mapsto -x_1 \quad \mbox{(or $x_2\mapsto -x_2)$}
$$
and all other coordinates fixed with $x_0\neq 0$.
\end{prop}
\begin{proof} Replacing $x_1^2$ and $x_1^4$ in the defining equations for
  $\tilde{V}_{24}$ by $dx_1^2$ and $d^2x_1^4$, respectively, we get a
  system of equations that give rise to $\tilde{V}_{24,d}$ isomorphic to
  $\tilde{V}_{24}$ over $\QQ(\sqrt{d})$. The corresponding involution
  $\iota$ is clearly as stated. A holomorphic $3$-form on $V_{24}$ is given
  by 
$$
\Omega=\displaystyle\frac{dx_3\wedge dx_4\wedge dx_5}{8x_1x_2^3-8x_1^3x_2}
$$
and under the involution $x_1\mapsto -x_1$, $\Omega$ is mapped to $-\Omega$.
\end{proof}

\subsection{The rigid Calabi--Yau threefold of van Geemen
and Nygaard} Another interesting example is the case of the
rigid Calabi--Yau threefold of van Geemen and Nygaard.
In \cite{vgn}, van Geemen and Nygaard gave an example of a
rigid Calabi-Yau threefold defined over $\QQ$: Let
$Y\subset\PP^7$ be the complete intersection of the four
quadrics:
$$
\begin{array}{ccl}
y_0^2 &=& x_0^2 + x_1^2 + x_2^2 + x_3^2 \\
y_1^2 &=& x_0^2 - x_1^2 + x_2^2 - x_3^2 \\
y_2^2 &=& x_0^2 + x_1^2 - x_2^2 - x_3^2 \\
y_3^2 &=& x_0^2 - x_1^2 - x_2^2 + x_3^2 .
\end{array}
$$

The variety $Y$ has $96$ isolated singularities, which are ordinary
double points.  Let $X$ be a (small) blow-up of $Y$ along its singular
locus.  (A recent article of Freitag and Salvati-Manni \cite{fsm}
asserts that $Y$ admits a resolution that is a projective Calabi--Yau
threefold, $X$.) Then $X$ is a rigid Calabi--Yau threefold over
$\QQ$. Its attached newform is the unique newform of weight $4$ on
$\Gamma_0(8)$, cf.\ \cite{vgn}, Theorem 2.4. Notice that there is a
misprint in the equations on p.\ 56 of that paper: In the second
equation $x_3^2$ should occur with a minus sign as above rather than a
plus sign as in \cite{vgn}, p.\ 56. This is evident from the theta
relations on p.\ 54 of \cite{vgn}.

Again, we instantly see the existence of twists of $X$ via replacing
$x_0^2$ by $dx_0^2$ in the above equations. Thus:

\begin{prop} For any non-square $d\in\QQ^{\times}$ the above rigid
  Calabi--Yau threefold $X$ has a twist $X_d$ by $d$. The
  corresponding involution $\iota$ is given by
$$ x_0\mapsto -x_0
$$
and all other coordinates fixed.
\end{prop}

\begin{proof} The only thing we need to check is whether a holomorphic
  $3$-form is send by $\iota$ to $-\Omega$. Let
$$
f_0:=y_0^2-(x_0^2+x_1^2+x_2^2+x_3^2),
$$
and similarly, define $f_1, f_2$ and $f_3$ by the second, third and
the fourth equation, respectively. Then $\Omega$ may be given by
$$
\Omega=\frac{dx_0\wedge dx_2\wedge dx_3}{D}
$$
where
$$
D= \det \left(\frac{\partial f_i}{\partial y_j}\right)_{0\leq i, j\leq
  3} =2^4y_0y_1y_2y_3.
$$

Thus the involution given by $\iota: x_0\mapsto -x_0$ and fixing all
other coordinates will send $\Omega$ to $-\Omega$.
\end{proof}

\section{Remarks on the levels of twists} Suppose that $d\in\ZZ$ is
squarefree and suppose that $f$ is a newform of level $N$. One may ask
about the level of the twisted newform $f_d$. Viewed from the Galois
representation side, this amounts to asking for the conductor of $\rho
\otimes \chi$ where $\rho$ has conductor $N$ and $\chi$ is a
(quadratic) character of conductor $D$, say (so $D$ divides $4d$ in
the above setup).

As is well-known, the answer is $ND^2$ if $(N,D)=1$. (This can be
proved either via basic theory of conductors, or, alternatively, more
directly via the theory of modular forms.) When $(N,D) > 1$, however,
the question has no simple answer: the level of the twisted
representation will depend heavily on the behavior of the
representation $\rho$ at inertia groups over common prime divisors of
$N$ and $D$. Nevertheless, see \cite{kiming}, \cite{kiming_verrill}
for the cases where $\rho$ has a `small' image.

For the concrete examples of twisting that we have discussed in this
paper, a more modest question can be asked, namely whether the newform
that we start with is `twist minimal' or not, i.e., whether it has the
lowest level among all of its quadratic twists. This question can be
easily answered with a little computation.

Let us for example consider the case of the Schoen quintic. By
Proposition 5.3 of \cite{S_quintic} (and its proof) one has that the
attached newform $f$ is of weight $4$ on $\Gamma_0(25)$ given as an
explicit linear combination of certain $\eta$-products. From this
explicit description of $f$ one computes that the coefficient of $q^2$
in its $q$-expansion is $-84$.

On the other hand, there is a unique newform $f_0$ of weight $4$ on
$\Gamma_0(5)$, namely $f_0(z) := \eta(z)^4 \eta(5z)^4$. We compute
that the coefficient of $q^2$ in the $q$-expansion of $f_0$ is $-4$.

Since $-4 \neq \pm (-84)$ we can deduce that $f_0$ is not a twist of
$f$ and hence that $f$ is twist minimal. In particular, the unique
newform $f_0$ of level $4$ is not the form attached to a twist of the
Schoen quintic. Is it attached to any Calabi--Yau threefold?

\section{Final remarks}\label{final}

\subsection{} As we implied in the introduction, the main contribution
of this paper is to put focus on the question whether any rigid Calabi--Yau
threefold over $\QQ$ has a twist by $d$ for any non-square $d\in\QQ$. In
contrast with the classical situation involving elliptic curves, the
question for rigid Calabi--Yau threefolds over $\QQ$ seems genuinely more
difficult. One difference is that one does not know in general the
automorphism group of a rigid Calabi--Yau threefold over $\QQ$. Another
point is the poor understanding of the conductor of a rigid Calabi--Yau
threefold over $\QQ$, i.e., the level of its associated newform.

For many other cases than the ones we have considered here, the existence
of quadratic twists of a given Calabi--Yau threefold over $\QQ$ can be
shown along the same lines as above, i.e., inspection of the defining
equation(s) combined with an application of Theorem \ref{thm_twist}.  For
instance, one can try to show the existence of quadratic twists for many
rigid Calabi--Yau threefolds over $\QQ$ discussed in Meyer \cite{Meyer}.

However, in some cases the question does not seem as easy. For instance,
does the rigid Calabi--Yau threefold of Hirzebruch (Theorem 5.11 in Yui
\cite{Y03}) admit quadratic twists? Let $X_0$ be the quintic threefold
defined over $\QQ$ by the equation $F(x,y)-F(u,w)=0$ where 
$$
F(x,y)=\left(x+\frac{1}{2}\right)\left(y^4-y^2(2x^2-2x+1)+
\frac{1}{5}(x^2+x-1)^2\right). 
$$

Then $X_0$ has $126$ nodes (ordinary double points) as only singularities.
Let $X$ be the blow up of $X_0$ along its singular locus.  Then $X$ is a
rigid Calabi--Yau threefold defined over $\QQ$ with the Euler
characteristic $306$. The map sending $y$ to $-y$ gives rise to an
involution on $X$ and this raises the obvious question of whether this
induces a non-trivial involution on $H^3$ so that we get a quadratic twist
of $X$ by replacing $y^2$ by $dy^2$.

\subsection{} Does there exist a rigid Calabi--Yau threefold $X$ defined
over $\QQ$ and a non-square rational $d$ such that $X$ does not have a
quadratic twist $X_d$ by $d$? 

Perhaps, in order to approach this question, one needs to loosen the
definition of `quadratic twist by $d$' so that the existence of $X_d$
becomes {\it equivalent to} (rather than just implying) the existence of a
rigid Calabi--Yau threefold over $\QQ$ whose attached $\ell$-adic Galois
representation (for some prime $\ell$) is the twist by the quadratic
character of $\QQ(\sqrt{d})/\QQ$ of the $\ell$-adic representation attached
to the original threefold. Maybe this is possible by considering, more
generally, algebraic correspondences rather than isomorphisms in the
setting of Theorem \ref{thm_twist}. 

Calabi--Yau threefolds, even rigid ones, in general, may not have
involutions; even when they do, it might be rather difficult to find
one. So geometric realization of modular forms of weight $4$ on some
$\Gamma_0(N)$ with integral Fourier coefficients along our proposed
approach may in fact not be overly promising. But this remark once again
raises the question of when the kind of twisting that we have discussed in
this paper is possible. 

\subsection{} We should include here a description of the fixed point
set of the involution $\iota$ acting on a rigid Calabi--Yau threefold
(though we did not make use of it in the examples.) The following
observation is due to B. van Geemen.

Let $X$ be a Calabi--Yau threefold over $\QQ$. Let $\iota$ be an
involution acting on $X$. Then the fixed point set of $\iota$ on
$X$ is determined as follows.

Suppose that $p$ is a fixed point of $\iota$, then in suitable local
coordinates $z_i, i=1,2,3$, $z_i(p)=0$, and
$$
\iota(z_1, z_2, z_3)=(e_1z_1, e_2z_2, e_3z_3) \quad\mbox{where $e_i=\pm 1$},
$$
and the dimension of the fixed point set is thus the number of $e_i$'s
which are equal to $+1$ (so if $\iota\neq 1$, then at least one $e_i$
must be $-1$.)

If $\Omega$ is the nowhere vanishing holomorphic $3$-form in
$H^{3,0}(X)$, then in these coordinates,
$$
\Omega=f(z_1, z_2, z_3) dz_1\wedge dz_2\wedge dz_3\quad\mbox{and $f(0,0,0)\neq 0$.}
$$
But $f(0,0,0)\neq 0$ forces that
$$
f(e_1z_1, e_2z_2, e_3z_3)=+f(z_1, z_2, z_3),
$$ and hence
$$
\iota^*\Omega=e_1e_2e_3\Omega.
$$

>From this we see that if the fixed point locus consists of isolated
points and divisors, then all or exactly one of the $e_i$ are $-1$,
else two of the $e_i=+1$ and the one $-1$.

\section{Appendix}\label{appendix}
\begin{proof}[Proof of Lemma \ref{diff_form}] Permuting variables if
  necessary we may, and will, assume $i_0 = 0$ and $I = \{ 1,\ldots
  ,k\}$. We put $D:=D_I$ with this particular $I$. 

Let us first recall some general facts about the Griffiths residue map:
Suppose that $V$ is a smooth projective variety of dimension $n$ and that
$W$ is a smooth codimension $1$ subvariety, or, more generally, a normal
crossings divisor. The Griffiths residue map (see for instance the proof of
Prop.\ 8.32 in \cite{voisin}) is a surjective homomorphism 
$$
\Res : ~ \Omega_V^n(\log W) \rightarrow \Omega_W^{n-1}
$$
defined as follows: If $(U,(z_1,\ldots,z_n))$ is a complex chart
of $V$ such that $W$ is given locally by the equation $z_1=0$ then
$\Omega_V^n(\log W)_{\mid U}$ is the free sheaf of ${\mathcal O}_U$-modules
generated by 
$$
\alpha := \frac{dz_1}{z_1} \wedge dz_2 \wedge \ldots \wedge dz_n ,
$$
and $\Res$ is then defined locally by
$$
\Res (g\alpha) := (g dz_2 \wedge \ldots \wedge dz_n)_{U\cap W} .
$$

Notice that $(U\cap W , (z_1,\ldots,z_{n-1})_U)$ is a complex chart of $W$.

Now, for the proof of the Lemma, let us first consider the case $k=1$ where
we specialize the above to the situation $V=\PP^n$ and $W=Y$ the
hypersurface given by the equation $f_1=0$. Consider the open set $U_0$
where $x_0 \neq 0$, let $z_i := x_i/x_0$ on $U_0$, and let $F := f_1/x_0^t$
where $t$ is the degree of $f_1$. Then $\omega_n := dz_1 \wedge \ldots
\wedge dz_n$ is a generator of $\Omega_{U_0}^n$. 

The open subset $U_0'$ of $U_0$ where $\partial F/\partial z_1 \neq 0$
coincides with the open subset where $\partial f_1/\partial x_1 \neq 0$. On
$U_0'$ we have local coordinates $F,z_2,\ldots,z_n$. Since 
$$
dF = \sum_{i=1}^n (\partial F/\partial z_i) dz_i
$$
we find
$$
dF \wedge dz_2 \wedge \ldots \wedge dz_n = \left( \frac{\partial F}{\partial z_1} \right) \omega_n
$$
so that
$$
\Res (\frac{\omega_n}{F})_{\mid U_0'} = \Res ( \frac{dF \wedge dz_2 \wedge \ldots \wedge dz_n}{F \left( \frac{\partial F}{\partial z_1}\right)} )_{\mid U_0'} = \left( \frac{dz_2 \wedge \ldots \wedge dz_n}{\left( \frac{\partial F}{\partial z_1}\right)} \right)_{\mid U_0'}
$$
which coincides up to a power of $x_0$ with
$$
\frac{dx_2 \wedge \ldots \wedge dx_n}{\left( \frac{\partial f_1}{\partial x_1}\right)}
$$
on $U_0'$.

Since the residue is holomorphic on all of $Y$, this differential form on
$U_0'$ will extend holomorphically to all of $Y$. 
\smallskip

For the general case, one can argue inductively with respect to $k$: We see
$Y$ as the end of a chain $Y:=Y_k \subseteq \ldots \subseteq Y_1 \subseteq
\PP^n$ where each $Y_i$ is a codimension $1$ subvariety of the $Y_{i-1}$
defined by the equation $f_i=0$. Dividing by suitable powers of $x_0$ to
define $F_i$ from $f_i$ as above and retaining local coordinates $z_i :=
x_i/x_0$ for $i>0$, the conclusion is now that we get a holomorphic
$3$-form on $U_0$ (where $x_0\neq 0$) by taking the residue of the form
$\omega_n/(F_1\cdots F_k)$. 

If we define ${\mathcal D} := \det \left( \frac{\partial F_i}{\partial z_j}
\right)_{1\le i,j\le k}$ then 
$$
dF_1 \wedge \ldots dF_k \wedge dz_{k+1} \wedge \ldots \wedge dz_n =
{\mathcal D} \cdot dz_1 \wedge \ldots \wedge dz_n 
$$
as $dF_j = \sum_{i=1}^n (\partial F_j/\partial z_i) dz_i$, and by the
definition of the determinant and alternating property of wedge
products. Redefining $U_0'$ as the open subset of $U_0$ where ${\mathcal D}
\neq 0$ then $U_0'$ coincides with the open subset of $U_0$ where $D\neq 0$
as $D$ differs from ${\mathcal D}$ by a power of $x_0$. 

Thus, on $U_0'$ we can compute the above residue:
\begin{eqnarray*}
\Res( \frac{\omega_n}{F_1\cdots F_k} )_{\mid U_0'} &=& \Res ( \frac{dF_1
  \wedge \ldots dF_k \wedge dz_{k+1} \wedge \ldots \wedge dz_n}{F_1\cdots
  F_k {\mathcal D}} )_{\mid U_0'} \\ 
&=& \left( \frac{dz_{k+1} \wedge \ldots \wedge dz_n}{{\mathcal D}}
  \right)_{\mid U_0'} 
\end{eqnarray*}
which coincides with
$$
\frac{dx_{k+1} \wedge \ldots dx_n}{D}
$$
up to a power of $x_0$ on $U_0'$.

Again this form extends to all of $Y$ for the same reasons as in the case
$k=1$. 
\medskip

Now suppose that $Y$ has a crepant resolution $X$ that is Calabi-Yau
variety of dimension $\dim X\leq 3$. There is then a surjective map
$\Omega^3_X \rightarrow \Omega_Y^3$. Hence the holomorphic $3$-form on $Y$
that we constructed above on $Y$ extends to a holomorphic $3$-form on $X$. 
\end{proof}

\section{Acknowledgments}

We would like to thank Bert van Geemen, Ken-Ichiro Kimura, James D. Lewis,
and Matthias Sch\"utt for helpful discussions on this topic. 

We would also like to thank the referee for going meticulously through the
manu\-script and pointing out a number of inaccuracies in the previous
version. 

The first author thanks Queen's University for its hospitality when
this work was done and Colby College for research funding. The second
author thanks The Danish Council for Independent Research for
support. The third author thanks the Natural Sciences and Engineering
Research Council of Canada (NSERC) for their support.

\end{document}